\documentclass[12pt]{amsart}
\usepackage{amsfonts,latexsym,amsthm,amssymb,graphicx,mathdots}

\usepackage{hyperref}

\DeclareFontFamily{OT1}{rsfs}{}
\DeclareFontShape{OT1}{rsfs}{n}{it}{<-> rsfs10}{}
\DeclareMathAlphabet{\mathscr}{OT1}{rsfs}{n}{it}

\newtheorem{theorem}{Theorem}[section]
\newtheorem{lemma}[theorem]{Lemma}
\newtheorem{corol}[theorem]{Corollary}
\newtheorem{prop}[theorem]{Proposition}

{\theoremstyle{definition} }
{\theoremstyle{remark} \newtheorem{remark}[theorem]{Remark}
\newtheorem{example}[theorem]{Example}}

\numberwithin{equation}{section}

\newcommand{\Abb}{{\mathbb{A}}}
\newcommand{\Pbb}{{\mathbb{P}}}
\newcommand{\Tbb}{{\mathbb{T}}}
\newcommand{\Zbb}{{\mathbb{Z}}}
\newcommand{\one}{1\hskip-3.5pt1}
\newcommand{\cD}{{\mathscr{D}}}
\newcommand{\cJ}{{\mathcal{J}}}
\newcommand{\cO}{{\mathcal{O}}}
\newcommand{\csm}{{c_{\text{SM}}}}

\newcommand{\cma}{{c_{\text{Ma}}}}
\newcommand{\cmam}{{c^-_{\text{Ma}}}}
\newcommand{\uH}{{\underline{H}}}
\newcommand{\qede}{\hfill $\lrcorner$}

\DeclareMathOperator{\Eu}{Eu}
\DeclareMathOperator{\Sing}{Sing}
\DeclareMathOperator{\defe}{def}
\DeclareMathOperator{\codim}{codim}
\DeclareMathOperator{\Id}{Id}

\title{
Projective duality and a Chern-Mather involution
}
\author{Paolo Aluffi}
\address{
Mathematics Department, 
Florida State University,
Tallahassee FL 32306, U.S.A.
}
\email{aluffi@math.fsu.edu}

\begin{document}

\begin{abstract}
We observe that linear relations among Chern-Mather classes of
projective varieties are preserved by projective duality. We deduce
the existence of an explicit involution on a part of the Chow group of
projective space, encoding the effect of duality on Chern-Mather
classes. Applications include Pl\"ucker formulae, constraints on
self-dual varieties, generalizations to singular varieties of
classical formulas for the degree of the dual and the dual defect,
formulas for the Euclidean distance degree, and computations of
Chern-Mather classes and local Euler obstructions for cones.
\end{abstract}

\maketitle

%%%

\section{Introduction}\label{sec:intro}

The conormal space of a projective variety $V\subseteq \Pbb^n$ is
intimately related with both the dual variety $V^\vee\subseteq
\Pbb^{n\vee}$ and intersection-theoretic invariants of $V$,
specifically the Chern-Mather class of~$V$. Not surprisingly, formulas
for Chern-Mather classes arise in the study of dual varieties, and
invariants associated with the latter may be expressed in terms of the
former. The main result of this paper is a very direct expression of
this kinship: we will prove that the (push-forward to projective space
of the) Chern-Mather class of $V^\vee$ may be obtained by applying an
explicit involution to the Chern-Mather class of $V$. While
straightforward, this fact carries a remarkable amount of information:
we will present applications to Pl\"ucker formulae, self-dual
varieties, formulas for the dual defect of a variety, computations of
Chern-Mather classes and local Euler obstructions for cones, and more.

The main result of this note admits several equivalent
formulations. Let $V\subseteq \Pbb^n$ be a projective variety over an
algebraically closed field of characteristic~$0$. 
The Chern-Mather class of $V$, introduced by
R.~MacPherson (\cite{MR0361141}), is a generalization to arbitrary
varieties of the notion of total Chern class of the tangent bundle of
a nonsingular variety. This note will be concerned with the
push-forward to $\Pbb^n$ of the Chern-Mather class; we denote this
push-forward by $\cma(V)$, and let $\cmam(V)=(-1)^{\dim V}
\cma(V)$. The first formulation expresses a basic compatibility
property of this signed Chern-Mather class under projective duality.

\begin{theorem}\label{thm:main}
Let $V_1,\dots, V_r\subsetneq \Pbb^n$ be proper closed subvarieties
and $a_1,\dots, a_r$ integers such that
\[
\sum_{k=1}^r a_k\, \cmam(V_k)=0
\]
in $A_*\Pbb^n$. Then
\[
\sum_{k=1}^r a_k\, \cmam(V_k^\vee)=0
\]
in $A_*\Pbb^{n\vee}$.
\end{theorem}

This result implies that we can define a `duality' map on rational
equivalence classes $\alpha\in A^*\Pbb^n$ supported in dimension~$<n$:
write $\alpha$ as a linear combination of signed Chern-Mather classes
of proper subvarieties, $\alpha=\sum_i a_i\, \cmam(V_i)$; and then
define $\alpha^\vee$ to be $\sum_i a_i\, \cmam(V_i^\vee)$. This
operation is well-defined by Theorem~\ref{thm:main}, maps $\cmam(V)$
to $\cmam(V^\vee)$, and is uniquely determined by this property.  The
other formulations of the result aim at making this duality map more
explicit and effectively computable.

For the second formulation, let $H$ denote the hyperplane class and
note that the classes
\begin{equation}\label{eqn:cmamp}
\{\Pbb^k\}:=\cmam(\Pbb^k)=(-1)^k (1+H)^{k+1} H^{n-k}\cap
      [\Pbb^n]\quad,\quad k=0,\dots,n
\end{equation}
form an additive basis of $A_*\Pbb^n$, so that every rational
equivalence class may be written as a combination of these classes.

\begin{theorem}\label{thm:basis}
Let $V\subsetneq \Pbb^n$ be a projective variety. If
\[
\cmam(V) = \sum_{i=0}^{n-1} a_i\, \{\Pbb^i\}\quad,
\]
then
\[
\cmam(V^\vee) = \sum_{i=0}^{n-1} a_{n-1-i}\, \{\Pbb^i\}\quad.
\]
\end{theorem}

The integers $a_i$ may be expressed as certain explicit combinations
of the degrees of the components of $\cma(V)$: $a_i=\sum_{j=i}^{\dim
  V} \binom{j+1}{i+1} (-1)^{\dim V-j} \cma(V)_j$
(see~\S\ref{ss:coefs} and Proposition~\ref{prop:KKH}). In fact, 
$a_i$ equals the $i$-th `rank', or `polar class', of $V$ 
(Remark~\ref{rem:polardegs}). Theorem~\ref{thm:basis} amounts to 
an alternative viewpoint on the very classical theory of these ranks: 
the ranks are interpreted here as the
coefficients of the signed Chern-Mather class in an additive basis for the 
Chow group of the ambient variety (projective space) obtained 
by considering signed Chern-Mather classes of certain distinguished 
subvarieties (linear subspaces). It is tempting to guess that
a similar approach may yield tools analogous to the ranks in more
general settings. Also, the ranks are essentially obtained by applying
to the Chern-Mather class of a variety a change-of-basis matrix 
whose entries are the coefficients of the
(signed) Chern classes of the distinguished subvarieties in terms of their 
fundamental classes. The square of this `Chern matrix' is the identity 
(Proposition~\ref{prop:squid}), a simple but intriguing fact that is likely 
just a facet of a substantially more general result, and which seems (to us) 
less apparent from the classical point of view on ranks.

The third formulation of the main result amounts to an explicit form 
of the duality introduced above.  For a fixed $n$, consider the following
$\Zbb$-linear map defined on polynomials $p(t)\in \Zbb[t]$:
\[
p\mapsto \cJ_n(p), \quad p(t) \mapsto
p(-1-t)-p(-1)\big((1+t)^{n+1}-t^{n+1}\big)\quad.
\]
It is immediately verified that if $\deg p\le n$, then $\deg \cJ_n(p)\le n$; and that $\cJ_n^2(p)=p$ 
for polynomials $p(t)\in (t)$.

\begin{theorem}\label{thm:invo}
Let $V\subsetneq \Pbb^n$ be a projective variety. If
\[
\cmam(V)=q(H)\cap [\Pbb^n]
\]
for a polynomial $q$ of degree $\le n$ in the hyperplane class $H$ in
$\Pbb^n$, then
\[
\cmam(V^\vee)=\cJ_n(q)\cap [\Pbb^{n\vee}]\quad,
\]
where $\cJ_n(q)$ is viewed as a polynomial in the hyperplane class in
$\Pbb^{n\vee}$.
\end{theorem}

The equivalence of Theorems~\ref{thm:main}---\ref{thm:invo} is
verified in~\S\ref{sec:equiv}; and then Theorem~\ref{thm:main} is
proven in~\S\ref{sec:mainpf}. Corollary~\ref{corol:concla}, proven
along the way, has a direct application to the {\em Euclidean distance
  degree,\/} and this is presented in~\S\ref{ss:ED} as a
generalization of a result from~\cite{eudideg}.

Other applications illustrating the use of the main theorem are
presented in~\S\ref{sec:examples}.  With some exceptions, these
results are standard; our main goal here is to illustrate that the
simple `Chern-Mather involution' defined above leads to very efficient
proofs of these facts, often providing at the same time a
straightforward generalization to the singular
case. In~\S\ref{ss:Pluecker} we recover a general form of the
classical Pl\"ucker formula for plane curves, and give a transparent
derivation of B.~Teissier's generalization of this formula from a
result of R.~Piene on Chern-Mather classes.
(This is essentially a re-packaging of Piene's own proof 
in~\cite{MR1074588}. Also see~\cite[p.~356-7]{MR0568897} for an 
equally transparent derivation of a generalization of this formula.)
In~\S\ref{ss:selfdual} we obtain constraints for a variety to be
self-dual; for example, we show (Proposition~\ref{prop:hypcons}) that
the singular locus of a self-dual hypersurface of degree $d$
in~$\Pbb^n$ has dimension~$\ge \frac{n-3}2$.  This simple statement
may be new---we were not able to find it in the literature.
Proposition~\ref{prop:KKH} in~\S\ref{ss:dude} recovers and generalizes
to singular varieties the Katz-Kleiman-Holme formula for the `dual
defect' and degree of $V^\vee$. The formula we obtain is equivalent
to, but somewhat leaner than, the generalization given
in~\cite{MR2336689} (see Remark~\ref{rem:MT} for a comparison).

In~\S\ref{ss:cones} (Propositions~\ref{prop:conform},
\ref{prop:euler}, and \ref{prop:Piene}) we obtain formulas for the
Chern-Mather class of a cone and for the local Euler obstruction along
the vertex of a cone; these formulas also appear to be new, although
they would likely be straightforward consequences of other known results,
such as the expression of the local Euler obstruction as an alternating
sum of multiplicities of polar varieties from~\cite{MR634426}. For
instance, we prove that if $V$ is a cone over a variety~$W$, then the
local Euler obstruction of $V$ at a point $p$ on the vertex equals
\[
\Eu_V(p)=\sum_{j=0}^{\dim W} (-1)^j \cma(W)_j
\]
where $\cma(W)_j$ denotes the degree of the $j$-dimensional component
of the Chern-Mather class of $W$. Questions of this type do not
involve duality in their formulation, but duality offers an effective
tool to address them, relying directly on the Chern-Mather
involution. The proofs given here are direct and self-contained. \smallskip

An alternative treatment of the connection between the Chern-Mather
classes of a projective variety $V$ and of its dual $V^\vee$
(including the more general case of higher order duals in
Grassmannians) may be found in~\cite{MR1468918} (Theorem 4.5), with
focus on the relation between the corresponding constructible
functions via the {\em topological Radon transform\/} studied by
L.~Ernstr\"om, cf.~\cite{MR1301184}. (We should also mention the formalism
of \cite[\S9.7]{MR95g:58222}, particularly the `Fourier transform' of p.~403
and ff.) The approach taken in this
note is different: here we decompose the Chern-Mather class of a
variety as a linear combination of Chern-Mather classes of {\em a
  priori\/} unrelated varieties, such as linear subspaces; no such
decomposition holds at the level of constructible functions. Working
directly at the level of classes results in a coarser theory, but
affords a greater level of flexibility, which plays a key role in the
applications presented in~\S\ref{sec:examples}. The involution
introduced here is also a Radon-type transformation, but it is
different from either the `homological Radon transformation' or the
`homological Verdier Radon transformation'
of~\cite[\S3]{MR1468918}. It would be interesting to study precise
relationships between these notions.
\smallskip

{\em Acknowledgments.} 
The author thanks Giorgio Ottaviani for conversations that led him to
reconsider Chern-Mather classes; the applications to the Euclidean
distance degree given in~\S\ref{sec:proof} and~\S\ref{sec:examples}
are a direct offspring of those conversations. The author also thanks
Steven Kleiman and Ragni Piene for very useful comments on a 
previous version of this note.

The author's research is supported in part by the Simons foundation
and by NSA grant H98230-15-1-0027. The author is grateful to Caltech
for hospitality while this work was carried out.

%%%

\section{Proofs}\label{sec:proof}

\subsection{Theorems~\ref{thm:main}---\ref{thm:invo} are 
equivalent}\label{sec:equiv}
The fact that Theorem~\ref{thm:main} implies Theorem~\ref{thm:basis}
follows immediately from the fact that the dual of a dimension $i$
linear subspace of $\Pbb^n$ is a dimension $n-1-i$ linear subspace of
$\Pbb^{n\vee}$: if $\cmam(V)=\sum_{i=0}^{n-1} a_i\{\Pbb^i\}$, then
$\cmam(V)-\sum_{i=0}^{n-1}a_i \cmam(\Pbb^i)=0$, so $\cmam(V^\vee)
-\sum_{i=0}^{n-1}a_i \cmam(\Pbb^{n-1-i})=0$ by Theorem~\ref{thm:main},
and Theorem~\ref{thm:basis} follows.

To see that Theorem~\ref{thm:basis} implies Theorem~\ref{thm:invo},
observe that $ \{\Pbb^k\} = \pi_k(H)\cap [\Pbb^n] $ with
$\pi_k(H)=(-1)^k\left((1+H)^{k+1} H^{n-k}-H^{n+1}\right)$, a
polynomial in $H$ of degree $\le n$. If
$
\cmam(V) = q(H)\cap [\Pbb^n]
$
as in the statement of Theorem~\ref{thm:invo}, and $\cmam(V)=\sum_i
a_i \{\Pbb^i\}$, then $q(H) = \sum_i a_i \pi_i(H)$. Therefore
$\cJ_n(q)=\sum_i a_i \cJ_n(\pi_i)$ by linearity, while by
Theorem~\ref{thm:basis}
\[
\cmam(V^\vee) = \sum_i a_i \{\Pbb^{n-i-1}\}\quad.
\]
It suffices then to observe that $\cJ_n(\pi_k)\cap
[\Pbb^n]=\{\Pbb^{n-k-1}\}$, that is,
\[
\cJ_n\left((-1)^k\left((1+H)^{k+1} H^{n-k}-H^{n+1} \right)\right) 
= (-1)^{n-k-1}\left((1+H)^{n-k} H^{k-1}-H^{n+1}\right)\quad,
\]
and this is an immediate consequence of the definition of $\cJ_n$.

Finally, to verify that Theorem~\ref{thm:invo} implies
Theorem~\ref{thm:main}, assume as in the statement of
Theorem~\ref{thm:main} that $\sum_{k=1}^r a_k\, \cmam(V_k)=0$. If
$q_k\in (H)\subseteq \Abb_*(\Pbb^n)$ are the polynomials of degree
$\le n$ corresponding to the classes $\cmam(V_k)$, we have
\[
\sum_{k=1}^r a_k\,\cmam(V_k)=0 \implies \sum_{k=1}^r a_k\,q_k=0 \implies
\sum_{k=1}^r a_k\,\cJ_n(q_k)=0
\]
by linearity, and this implies $\sum_{k=1}^r a_k\,\cmam(V_k^\vee) = 0$
according to Theorem~\ref{thm:invo}, yielding Theorem~\ref{thm:main}.

\subsection{Proof of Theorem~\ref{thm:main}}\label{sec:mainpf}
We work over an algebraically closed field of characteristic~$0$. This
restriction on the characteristic is commonly adopted in dealing with
characteristic classes of singular varieties, and it is further needed
for Lemma~\ref{lem:1524} below.

The (total) {\em Chern-Mather\/} class of $V$ is defined by R.~MacPherson 
as the push-forward of the Chern class of the tautological bundle of the
Nash blow-up of $V$, \cite[\S2]{MR0361141}. By construction the tautological
bundle restricts to the pull-back of the tangent bundle over the nonsingular
part $V_\text{reg}$ of $V$, and it follows that the Chern-Mather class restricts 
to $c(TV_{\text{reg}})\cap [V_{\text{reg}}]$ on $V_{\text{reg}}$. In particular, the
Chern-Mather class agrees with $c(TV)\cap[V]$ if $V$ is nonsingular.
MacPherson's theory of Chern
classes for singular varieties is based on a natural trasformation
defined by associating the Chern-Mather class with the {\em local
  Euler obstruction,\/} cf.~\cite[\S3]{MR0361141}. (The local Euler obstruction
was independently introduced and studied by M.~Kashiwara in the
context of index theorems for holonomic $\cD$-modules, \cite{MR0368085}.)

We will use an alternative formulation, given by C.~Sabbah in terms of
the {\em conormal variety,\/} \cite[\S1.2.1]{MR804052}. For $V$ a
proper closed subvariety of $\Pbb^n$, we denote by
$\Pbb(C^*_V\Pbb^n)\subseteq \Pbb(T^*\Pbb^n)$ the projective conormal
variety of $V$, obtained as the Zariski closure of the
projectivization of the conormal bundle of the nonsingular part
$V_\text{red}$ of~$V$. Letting $g:\Pbb(C^*_V\Pbb^n)\to V\hookrightarrow
\Pbb^n$ be the projection followed by the inclusion in~$\Pbb^n$, the
Chern-Mather class of~$V$ pushes forward to the class
\[
\cma(V)=(-1)^{n-1-\dim V} c(T\Pbb^n)\cap g_*\left( c(\cO(1))^{-1}\cap
    [\Pbb (C^*_V\Pbb^n)]\right)
\]
in $\Abb_*\Pbb^n$ (cf.~\cite[\S1]{MR2002g:14005}). Here, $\cO(1)$ is
the restriction of the universal quotient bundle over $\Pbb(T^*\Pbb^n)$. 
If $V\overset \iota\hookrightarrow \Pbb^n$ is a {\em nonsingular\/}
variety, then $\cma(V)=\iota_* (c(TV)\cap [V])$.

As in~\S\ref{sec:intro} we introduce a sign according to the dimension
of $V$ and let
\[
\cmam(V)=(-1)^{n-1} c(T\Pbb^n)\cap g_*\left( c(\cO(1))^{-1}\cap [\Pbb
  (C^*_V\Pbb^n)]\right) \in\Abb_*\Pbb^n\quad.
\]
Now $\Pbb(T^*\Pbb^n)$ may be realized as the incidence correspondence
$I\subseteq \Pbb^n\times \Pbb^{n\vee}$, a divisor of class $H+h$
(denoting by $H$, resp., $h$ the hyperplane classes in $\Pbb^n$,
resp.,~$\Pbb^{n\vee}$, and their pull-backs).  It is easy to verify
that in this realization the universal bundle $\cO(1)$ agrees with the
normal bundle to $I$ in $\Pbb^n\times \Pbb^{n\vee}$. Further, the
projective conormal variety $\Pbb(C^*_V\Pbb^n)$ may then be identified
with the closure of the correspondence
\[
\Phi_V:=\overline{\{(p,H)\, |\, \text{$p\in V_\text{reg}$ and
    $H\supseteq \Tbb_p V$}\}}\quad.
\]
Therefore we may write
\begin{equation}\label{eqn:key1}
\cmam(V)=(-1)^{n-1} (1+H)^{n+1}\cap \pi_{1*}\left(\frac 1{1+H+h}\cap
     [\Phi_V]\right)
\end{equation}
where $\pi_{1*}$ is the restriction of the projection onto the first
factor $\Pbb^n\times \Pbb^{n\vee}\to \Pbb^n$.

Switching factors gives the canonical isomorphism
\begin{equation*}
\tag{*} \Pbb^n \times \Pbb^{n\vee} \cong \Pbb^{n\vee}\times \Pbb^n
\cong \Pbb^{n\vee}\times \Pbb^{n\vee \vee}\quad,
\end{equation*}
and we recall the following classical result, which holds under our 
blanket characteristic~$0$ hypothesis. (In fact, this reflexivity holds
as soon as the projection $\Phi_V \to V^\vee$ is generically smooth,
cf.~\cite[p.~169]{MR846021}.)

\begin{lemma}\label{lem:1524}
Via the identification (*), $\Phi_V=\Phi_{V^\vee}$.
\end{lemma}

\begin{proof}
See \cite{MR846021} or e.g.,~\cite[Theorem 1.7]{MR2113135}. 
\end{proof}

Let then $\Phi:=\Phi_V=\Phi_{V^\vee}$. The class of $\Phi$ in
$\Pbb^n\times \Pbb^{n\vee}$ determines $\cmam(V)$ by \eqref{eqn:key1},
and $\cmam(V^\vee)$ by the same token:
\begin{equation}\label{eqn:key2}
\cmam(V^\vee)=(-1)^{n-1} (1+h)^{n+1}\cap \pi_{2*}\left(\frac
1{1+H+h}\cap [\Phi]\right)
\end{equation}
where $\pi_2$ is the restriction of the projection onto the second
factor $\Pbb^n\times \Pbb^{n\vee}\to \Pbb^{n\vee}$.  In order to prove
Theorem~\ref{thm:main} it suffices to prove that, conversely, $[\Phi]$
is determined by the class $\cmam(V)$ and depends linearly on this
class.

This is a general fact, with no direct relation to Chern-Mather classes.

\begin{prop}\label{prop:key}
Let $\gamma=\sum_{i=0}^{n-1} \gamma_i\, [\Pbb^i]$ be a class in
$A_*\Pbb^n$. Then the unique class $\Gamma$ in $A_{n-1}I$ such that
\[
\gamma=(-1)^{n-1} (1+H)^{n+1}\cap \pi_{1*}\left(\frac 1{1+H+h}\cap
\Gamma\right)
\]
is
\[
\Gamma=\sum_{i=0}^{n-1} (-1)^i \gamma_i\, (H+h)^i H^{n-i}\cap [I]\quad.
\]
\end{prop}

\begin{proof}
Since $I\cong \Pbb T^*\Pbb^n$, every class $\Gamma\in A_{n-1} I$ may
be written uniquely as
\[
\Gamma=\sum_{i=0}^{n-1} a_i\, c_1(\cO(1))^i\cdot H^{n-i} \cap [I]\quad,
\]
by \cite[Theorem~3.3 (b)]{85k:14004}. Therefore, as a class in
$\Pbb^n\times \Pbb^{n\vee}$,
\[
\Gamma=\sum_{i=0}^{n-1} a_i\, (H+h)^{i+1}\cdot H^{n-i}\cap
           [\Pbb^n\times \Pbb^{n\vee}]\quad.
\]
It follows that the push-forward
\[
\pi_{1*}\left(\frac 1{1+H+h}\cap \Gamma\right)
\]
is the coefficient of $h^n$ in the series expansion of
\[
\sum_{i=0}^{n-1} a_i\, \frac{(H+h)^{i+1}\cdot H^{n-i}}{1+H+h}
\]
and this is easily computed to be
$
\sum_{i=0}^{n-1} a_i\, \frac{(-1)^{n-i-1} H^{n-i}}{(1+H)^{n+1}}\quad.
$
Therefore
\[
(-1)^{n-1} (1+H)^{n+1}\cap \pi_{1*}\left(\frac 1{1+H+h}\cap \Gamma\right)
=\sum_{i=0}^{n-1} (-1)^i a_i\, H^{n-i}\quad.
\]
Equating this class with $\gamma$ gives the statement. 
\end{proof}

Theorem~\ref{thm:main} is an easy consequence of
Proposition~\ref{prop:key}.

\begin{proof}[Proof of Theorem~\ref{thm:main}]
Assume $\sum_{k=1}^r a_k\,\cmam(V_k)=0$. By
Proposition~\ref{prop:key},
\[
\sum_{k=1}^r a_k\,[\Phi_k]=0
\]
where $\Phi_k=\Phi_{V_k}=\Phi_{V_k^\vee}$.  By~\eqref{eqn:key2} this
implies that $\sum_{k=1}^r a_k\,\cmam(V_k^\vee)=0$,
completing the proof.
\end{proof}

The explicit relation between the Chern-Mather class of $V$ and the
conormal cycle $\Phi=\Phi_V$ is the following.

\begin{corol}\label{corol:concla}
Let $V$ be a proper closed subvariety of $\Pbb^n$. Then the class of
the conormal variety of $V$ in $\Pbb^n\times \Pbb^{n\vee}$ is given by
\begin{equation}\label{eqn:concla}
[\Phi_V] = \sum_{j=0}^{n-1} (-1)^{\dim V+j}\cma(V)_j\, (H+h)^{j+1}
H^{n-j} \cap [\Pbb^n\times \Pbb^{n\vee}]
\end{equation}
where $\cma(V)_j$ denotes the degree of the $j$-dimensional component
of the Chern-Mather class of $V$ in $\Pbb^n$:
$\cma(V)=\sum_{j=1}^{n-1} \cma(V)_j\, [\Pbb^j]$.
\end{corol}

\begin{proof}
This follows from~\eqref{eqn:key1} and
Proposition~\ref{prop:key}. (Note the correction in the sign of the
coefficients to account for the difference between $\cma(V)$ and
$\cmam(V)$.)
\end{proof}

\subsection{The $\{\Pbb^i\}$ coefficients}\label{ss:coefs}
It is occasionally useful to extract explicit expressions for the
coefficients $a_i$ appearing in Theorem~\ref{thm:basis}. It is in fact
a simple matter to obtain Poincar\'e duals of the classes
$\{\Pbb^i\}$, and this yields expressions of the coefficients for any
given class $\alpha$ as intersection degrees of $\alpha$ with certain
explicit classes in $\Pbb^n$.

\begin{lemma}\label{lem:efa}
For $i=0,\dots,n$ let
\[
\{\Pbb^i\}^* = \frac{(-1)^i}{(1+H)^{i+2}} [\Pbb^{n-i}]\quad.
\]
Then for all $0\le i,j\le n$, $\int \{\Pbb^i\}^* \cdot
\{\Pbb^j\}=\delta_{ij}$.
\end{lemma}

\begin{proof}
The claim is that
\[
\int \frac{(-1)^i H^i}{(1+H)^{i+2}}\cap \{\Pbb^j\} = \delta_{ij}
\]
for $0\le i,j\le n$, i.e., that the degree of $H^n$ in the expansion
of
\[
\frac{(-1)^i H^i}{(1+H)^{i+2}}\cdot (-1)^j (1+H)^{j+1} H^{n-j}
=(-1)^{i+j} (1+H)^{j-i-1} H^{n+i-j}
\]
is the Kronecker delta $\delta_{ij}$; and this is immediate.
\end{proof}

\begin{prop}\label{prop:ai}
Let $\alpha=\sum_j \alpha_j [\Pbb^j]\in A_*\Pbb^n$, and assume
$\alpha = \sum_i a_i \{\Pbb^i\}$.
Then
\begin{equation}\label{eqn:ai}
a_i=\int \frac{(-1)^i H^i}{(1+H)^{i+2}}\cap \alpha
=\sum_{j=i}^{\dim V} \binom{j+1}{i+1} (-1)^j \alpha_j\quad.
\end{equation}
\end{prop}

\begin{proof}
By Lemma~\ref{lem:efa},
\[
a_i = \sum_{j=0}^{n-1} a_j \delta_{ij} = \sum_{j=0}^{n-1} a_j \int
\{\Pbb^i\}^*\cdot \{\Pbb^j\} =\int \{\Pbb^i\}^*\cdot \alpha\quad,
\]
giving the first equality. The second equality is obtained by
computing the coefficient of $H^n$ in $\frac{(-1)^i H^i}{(1+H)^{i+2}}
\sum_j \alpha_j H^j$.
\end{proof}

\begin{example}\label{exa:aiforhyp}
Let $V\overset \iota\hookrightarrow \Pbb^n$ be a nonsingular hypersurface
of degree~$d$. Then
\[
\cmam(V)=(-1)^{n-1} \iota_* c(TV)\cap [V] = \sum_{i=0}^{n-1} d(d-1)^{n-1-i} \{\Pbb^i\}\quad.
\]
Indeed,
\begin{equation}\label{eqn:virtan}
\iota_* c(TV)\cap [V] = \frac{(1+H)^{n+1}\, dH}{1+dH}\cap [\Pbb^n]\quad,
\end{equation}
so by Proposition~\ref{prop:ai} the coefficient of $\{\Pbb^i\}$ in
$\cmam(V)$ equals the coefficient of $H^n$ in the expansion of
\[
(-1)^{n-1} \frac{(-1)^i H^i}{(1+H)^{i+2}}\cdot \frac{(1+H)^{n+1}\,
  dH}{1+dH} = (-1)^{n-i-1} \frac{(1+H)^{n-1-i}\, dH^{i+1}}{1+dH}\quad,
\]
i.e., the coefficient of $H^{n-i-1}$ in
\[
(-1)^{n-i-1}d\,\frac{(1+H)^{n-1-i}}{1+dH}\quad.
\]
By~\cite{onepage}, this equals the coefficient of $H^{n-i-1}$ in 
\[
(-1)^{n-i-1}d\,(1+(1-d)H)^{n-1-i}\quad,
\]
and this is $d(d-1)^{n-1-i}$.

More generally, we see that if $V$ is a hypersurface in $\Pbb^n$ and
$\dim \Sing V\le r$, and $\cmam=\sum_{i=0}^{n-1} a_i \{\Pbb^i\}$, then 
\[
a_i=d(d-1)^{n-1-i}\quad \text{for $i=r+1,\dots, n-1$.}
\]
Indeed, the Chern-Mather class of $V$ restricts to 
$c(TV_{\text{reg}})\cap [V_{\text{reg}}]$ on the nonsingular part 
$V_{\text{reg}}$, so it agrees with the Chern class of the virtual tangent
bundle of $V$ (given by the right-hand side of~\eqref{eqn:virtan}) in 
dimension $>r$.
\qede\end{example}

For instance, this implies that the dual of a nonsingular hypersurface $V$
of degree $d>1$ is a hypersurface of degree~$d(d-1)^{n-1}$: indeed,
the degree of $V^\vee$ is the coefficient of $\{\Pbb^{n-1}\}$ in
$\cmam(V^\vee)$, hence (by Theorem~\ref{thm:basis}) the coefficient of
$\{\Pbb^0\}$ in $\cmam(V)$, hence (by Example~\ref{exa:aiforhyp}) it
equals $d(d-1)^{n-1}$.

\begin{remark}\label{rem:polardegs}
In terms of conormal spaces, the equality $\cmam(V)=\sum_i a_i\{\Pbb^i\}$
is equivalent to $[\Phi_V]=\sum_i a_i [\Phi_{\Pbb^i}]$. Now
$\Phi_{\Pbb^i}=\Pbb^i\times \Pbb^{n-1-i}$ as a subvariety of $\Pbb^n\times \Pbb^n$;
thus
\[
[\Phi_V] = a_0 H^n h + \cdots + a_{n-1} H h^n\quad.
\]
Classically, the $a_i$ are known as the {\em ranks\/} (or {\em polar classes\/}) of $V$.
\qede\end{remark}

By Proposition~\ref{prop:ai}, the ranks of a projective variety $V\subsetneq
\Pbb^n$ may be obtained
from $\cmam(V)=\sum_j \alpha_j [\Pbb^j]$ by multiplying the $n\times n$
matrix $M$ with entries
\[
M_{ij} = (-1)^{j-1} \binom ji	
\]
by the column vector $(\alpha_0,\dots, \alpha_{n-1})^t$.
We can now observe that this matrix is {\em also\/} the matrix expressing the
coefficients $\alpha_j$ of the signed $\cmam(V)$ in terms of the ranks~$a_i$,
that is, the $n\times n$ matrix whose $(i,j)$ entry is the coefficient of 
$[\Pbb^{i-1}]$ in $\cmam(\Pbb^{j-1})$. In other words, this `matrix of Chern
coefficients' $M$ satisfies the following identity.

\begin{prop}\label{prop:squid}
For all $n$, $M^2=\Id_n$.
\end{prop}

\begin{proof}
For all $1\le i,j \le n$,
\[
\sum_{k=1}^n (-1)^{k-1}(-1)^{j-1} \binom ik  \binom kj
=\sum_{k=1}^n (-1)^{k-j} \binom {i-j}{k-j} \binom ij = 
\delta_{ij} \binom ij = \delta_{ij}
\]
as needed.
\end{proof}

This intriguing fact appears to admit a nontrivial generalization to setting
of flag varieties, which will be discussed elsewhere.

\subsection{The Euclidean distance degree}\label{ss:ED}
The {\em Euclidean distance degree\/} of a variety is the number of
critical points of the squared distance to a general point outside the
variety. It is proven in~\cite[Theorem~5.4]{eudideg} that if the
incidence variety $\Phi=\Phi_V$, i.e., the projective conormal variety
$\Pbb (C^*_V\Pbb^n)$, does not meet the diagonal $\Delta(\Pbb^n) \subseteq
\Pbb^n\times \Pbb^n$, then the Euclidean distance degree of $V$ equals
the sum
\begin{equation}\label{eqn:eudeg}
\delta_0(V) + \cdots + \delta_{n-1}(V)\quad,
\end{equation}
where $\delta_i(V)$ are the polar degrees of $V$ (and hence equal the
coefficients~$a_i$ computed in~\S\ref{ss:coefs}, cf.~Remark~\ref{rem:polardegs}):
\[
[\Phi_V]=\delta_0(V) H^n h+ \cdots + \delta_{n-1}(V)H h^n\quad.
\]
The quantity \eqref{eqn:eudeg} is easy to compute in term of the
Chern-Mather class of $V$, by means of Corollary~\ref{corol:concla}.

\begin{prop}\label{prop:otta}
Let $V$ be a proper subvariety of $\Pbb^n$ of dimension $m$. Then
$\delta_i(V)=0$ for $i>m$, and
\[
\delta_0(V) + \cdots + \delta_m(V)=
\sum_{j=0}^m (-1)^{m+j} \cma(V)_j\, (2^{j+1}-1)\quad,
\]
where $\cma(V)_j$ is the degree of the $j$-dimensional component of
the Chern-Mather class of~$V$ in $\Pbb^n$.
\end{prop}

\begin{proof}
Formally set $H=h=1$ in the formula for $[\Phi_V]$ obtained in
Corollary~\ref{corol:concla}.  (The fact that $H^{n+1}=0$ accounts for
the `$-1$'.)  The vanishing of $\delta_i(V)$ for $i>m$ is also
immediate from~\eqref{eqn:concla}.
\end{proof}

Proposition~\ref{prop:otta} generalizes to arbitrarily singular varieties the 
formula appearing in Theorem~5.8 in~\cite{eudideg}: in the singular case,
Chern-Mather classes play in this formula the same role as ordinary 
Chern classes in the nonsingular case.
The right-hand side of this formula computes the Euclidean distance
degree of any (nonsingular or otherwise) closed subvariety $V\subsetneq
\Pbb^n$ such that $\Phi_V$ does not meet the diagonal
$\Delta(\Pbb^n)\subseteq\Pbb^n\times \Pbb^{n\vee}$.

Other general formulas for the Euclidean distance degree in terms of degrees 
of polar classes are given in~\cite[\S5]{MR3335572}. The precise relation between 
polar classes and Chern-Mather classes is studied in~\cite{MR1074588}.

%%%

\section{Examples and applications}\label{sec:examples}

\subsection{Pl\"ucker formulae}\label{ss:Pluecker}
Several formulas of Pl\"ucker type may be recast in terms of
Theorem~\ref{thm:invo}; we recover the classical form of the Pl\"ucker
formulae for plane curves as an illustration.

Let $C$ be a reduced, irreducible
plane curve of degree $d\ge 2$,
and let $p_1,\dots, p_r$ be the singular points of $C$. Let $m_i$,
resp., $\mu_i$ be the multiplicity, resp., Milnor number of $C$ at
$p_i$. A standard computation gives
\begin{equation}\label{eqn:cmaC}
\cmam(C) = -(dH+(3d-d^2+\sum_i (\mu_i+m_i-1)) H^2) \cap [\Pbb^2]\quad.
\end{equation}
(Proof: The Chern-Mather class of $C$ is the image of the local Euler
obstruction $\Eu_C$ via MacPherson's natural transformation. Since the
local Euler obstruction of a curve equals the multiplicity function,
we have $ \Eu_C = \one_C + \sum_i (m_i-1) \one_{p_i} $, and therefore
\[
\cma(C) = d[\Pbb^1]+\chi(C) [\Pbb^0]+ \sum_i (m_i-1)[p_i] = 
(d H + (\chi(C) +\sum_i(m_i-1)) H^2)\cap [\Pbb^2]
\]
after push-forward to $\Pbb^2$. The Euler characteristic of $C$ equals
the Euler characteristic of a nonsingular curve of degree $d$
corrected by the presence of the singular points; each points
contributes by its Milnor number, and this gives the stated formula.)

Let $\rho_i=\mu_i+m_i-1$. For example, $\rho_i=2$ if $p_i$ is a node
and $\rho_i=3$ if $p_i$ is a cusp.

\begin{prop}
With notation as above:
\begin{equation}\label{eqn:Cdu}
\cmam(C^\vee) = -\left((d(d-1)-\sum_i \rho_i)H+(d(2d-3)-2\sum_i \rho_i) H^2\right)
\cap [\Pbb^2]\quad.
\end{equation}
\end{prop}

\begin{proof}
Apply the involution in Theorem~\ref{thm:invo} to~\eqref{eqn:cmaC}.
\end{proof}

As an immediate consequence, we recover a (well-known) general form of
the classical Pl\"ucker formula for plane curves.

\begin{corol}
$
\deg C^\vee = d(d-1) - \sum_i \rho_i
$.
\end{corol}

The coefficient of $H^2$ in~\eqref{eqn:Cdu} also gives some
information.  Let $q_1,\dots, q_s$ be the singular points of $C^\vee$,
and $\rho^\vee_j$, $j=1,\dots, s$, the corresponding $\rho$
numbers. Then~\eqref{eqn:Cdu} implies that
\[
\sum_j \rho^\vee_j = R^2 -(2d^2-2d-1) R+ d^3(d-2)
\]
where $R=\sum_i \rho_i$.
The same conclusion may be drawn by applying biduality. For example,
if $C$ is nonsingular, so that $R=0$, and $C$ has $b$ bitangents and
$f$ inflection points (counted with appropriate multiplicities) this
shows that $2b+3f=d^3(d-2)$.

More generally, R.~Piene evaluates the contribution to the
Chern-Mather class of a hypersurface $X$ in $\Pbb^n$ due to isolated
singularities $p_i$ (\cite[p.~25]{MR1074588}) :
\begin{equation}\label{eqn:PieneMather}
\cma(X) = \cma(X_g) + \sum_i (-1)^n(\mu_X(p_i)+\mu_{X\cap
  \uH_i}(p_i))[\Pbb^0]
\end{equation}
where $X_g$ is a nonsingular hypersurface of the same degree as $X$
(so that $\cma(X_g)$ equals the push-forward of the ordinary total
Chern class of $X_g$) and $\uH_i$ is a general hyperplane through
$p_i$. Therefore, with notation as in~\S\ref{sec:intro} we have
\[
\cmam(X) = \cmam(X_g) - \sum_i (\mu_X(p_i)+\mu_{X\cap
  \uH_i}(p_i))\{\Pbb^0\}
\]
and Theorem~\ref{thm:main} implies
\[
\cmam(X^\vee) = \cmam(X_g^\vee) 
- \sum_i (\mu_X(p_i) + \mu_{X\cap \uH_i}(p_i))\{\Pbb^{n-1}\}\quad.
\]
Reading off the terms of dimension~$n-1$ in this identity
(cf.~Example~\ref{exa:aiforhyp}) yields
\begin{equation}\label{eqn:TeiPlu}
\deg(X^\vee) = d(d-1)^{n-1} - \sum_i (\mu_X(p_i) + \mu_{X\cap
  \uH_i}(p_i))\quad,
\end{equation}
that is, B.~Teissier's generalized Pl\"ucker
formula~\cite[II.3]{TeisRes}, \cite[p.~16]{MR1074588}.  Of course the
nontrivial input here is Piene's
generalization~\eqref{eqn:PieneMather} of the basic
formula~\eqref{eqn:cmaC}; we are simply observing that
Theorem~\ref{thm:main} makes the derivation of~\eqref{eqn:TeiPlu}
from~\eqref{eqn:PieneMather} particularly transparent.  Teissier has
shown that $\mu_X(p) + \mu_{X\cap \uH_i}(p)$ equals the multiplicity
of the Jacobian ideal at $p$ (in fact this is an ingredient in Piene's
proof of~\eqref{eqn:PieneMather}), and S.~Kleiman generalized the
Teissier-Pl\"ucker formula to arbitrary varieties with isolated
singularities by replacing this multiplicity with the Buchsbaum-Rim
multiplicity of the Jacobian matrix, \cite[Theorem~2]{MR1272702}. It
would be interesting to interpret Kleiman's formula as a computation
of a Chern-Mather class; for complete intersections, the Buchsbaum-Rim
multiplicity at $p$ equals $\mu_X(p) + \mu_{X\cap \uH_i}(p)$
(see \cite[(3.3.1)]{MR1702106}).

The same invariants appearing in the Pl\"ucker formula for curves
allow us to generalize another result from~\cite{eudideg}, by applying
Proposition~\ref{prop:otta} to expression~\eqref{eqn:cmaC} (or
equivalently~\eqref{eqn:Cdu}).

\begin{corol}\label{corol:otta}
Let $C\subseteq \Pbb^n$ be a reduced curve of degree~$d$. Assume
that $C$ meets the isotropic quadric transversally, and let $R=\sum_i
\rho_i$ be the invariant associated with a general plane projection of
$C$. Then the Euclidean distance degree of $C$ is given by
\[
\text{EDdegree}(C)=d^2-R\quad.
\]
\end{corol}

Here we also use the fact that the Chern-Mather class is preserved by
a general projection (\cite[Corollaire, p.~20]{MR1074588}; and
cf.~\cite[Corollary 6.1]{eudideg}). The `isotropic quadric' in $\Pbb^n$
is the quadric with equation $x_0^2+\cdots + x_n^2=0$. 
For a discussion of the requirement that $C$ meets it
transversally see \cite{eudideg}, particularly the comment
following~Theorem~5.4. Corollary~\ref{corol:otta}
generalizes~\cite[Corollary~5.9]{eudideg} to singular curves.

\subsection{Self-dual varieties}\label{ss:selfdual}
A variety $V\subseteq \Pbb^n$ is {\em self-dual\/} if there exists an
isomorphism $\Pbb^n\to\Pbb^{n\vee}$ restricting to an isomorphism 
$V\to V^\vee$.  Theorem~\ref{thm:invo} implies immediately 
that the Chern-Mather class of a self-dual plane curve is determined 
by its degree.

\begin{prop}
If $C$ is a self-dual curve of degree $d$ in $\Pbb^2$, then
\[
\cma(C)=d[\Pbb^1]+d[\Pbb^0]\quad.
\]
\end{prop}

\begin{proof}
If $\cma(C)=(dH+eH^2)\cap [\Pbb^2]$ and $C$ is self-dual, then
$\cJ_2(dH+eH^2) =dH+eH^2$ by Theorem~\ref{thm:invo}. By definition,
$\cJ_2(dH+eH^2)=(2d-e)H+ (3d-2e)H^2$. Thus necessarily $d=e$ for
self-dual curves.
\end{proof}

Applying Proposition~\ref{prop:otta}, we see that the Euclidean
distance degree of a self-dual curve of degree $d$ (meeting the
isotropic conic transversally) is $2d$.  There are self-dual curves of
all degrees $d\ge 2$: for all $1\le k\le d-1$, the curve $y^d = x^k
z^{d-k}$ is self-dual.

Analogous results can be obtained e.g., for surfaces in $\Pbb^3$.
\begin{prop}\label{prop:sdsurf}
Let $S\subseteq \Pbb^3$ be a self-dual surface of degree~$d$. Then
\[
\cma(S)=d[\Pbb^2]+e[\Pbb^1]+2(e-d)[\Pbb^0]
\]
for some integer $e$.
\end{prop}

\begin{proof}
According to Theorem~1.3, if $\cma(S)=(dH+eH^2+fH^3)\cap [\Pbb^3]$,
then
\[
\cma(S^\vee) = ((3d-2e+f)H+(6d-5e+3f)H^2+(4d-4e+3f)H^3)\cap [\Pbb^3]
\]
(if $3d-2e+f\ne 0$). Equating these two expressions determines $f$.
\end{proof}

If $S$ has isolated singularities, then necessarily $e=-d(d-4)$, so
this class equals
\begin{equation}\label{eqn:sdsur}
\cma(S)=d[\Pbb^2] - d(d-4) [\Pbb^1] -2d(d-3) [\Pbb^0]\quad.
\end{equation}
Examples include the nonsingular quadric, the cubic surface $x_0^3=x_1 x_2
x_3$, and the Kummer quartic surface (\cite[p.~784]{MR507725}).
According to~\eqref{eqn:sdsur}, any self-dual quartic surface with
isolated singularities must have Chern-Mather class equal to
$4[\Pbb^2]-8[\Pbb^0]$. (Another candidate is a quartic surface with
$7$ ordinary double points and $3$ singularities of type $A_3$
constructed in~\cite{naruki}.)  This is a direct consequence of
Theorem~\ref{thm:invo}, and does not require any specific knowledge of
these surfaces. Similarly, no further work is needed to establish the
following result.

\begin{corol}
The Euclidean distance degree of a self-dual surface of degree~$d$ with 
isolated singularities in $\Pbb^3$, transversal to the isotropic quadric, 
is~$d^2+d$.
\end{corol}

(Use Proposition~\ref{prop:otta}.) With notation as in Proposition~\ref{prop:sdsurf},
the Euclidean distance degree of a self-dual surface with arbitrary singularities 
in $\Pbb^3$ is $5d-e$. The Euclidean distance degree of a
general nonsingular surface of degree~$d$ in $\Pbb^3$ is $d^3-d^2+d$.

For more refined considerations, we can again invoke Piene's result
\eqref{eqn:PieneMather}: an isolated singularity $p$ corrects the
Chern-Mather class of $S$ by $\rho(p)=\mu_S(p) +\mu_{S\cap \uH}(p)$,
where $\uH$ is a general plane through $p$.  For example, $\rho=2$ for
an ordinary double point.

\begin{prop}
Let $S\subseteq \Pbb^3$ be a self-dual surface of degree $d$ with
isolated singularities $p_1,\dots, p_s$. Then $\sum_j \rho(p_j) =
d^2(d-2)$.
\end{prop}

\begin{proof}
We have
\[
\cma(S) = d[\Pbb^2]+d(4-d) [\Pbb^1] + (d(d^2-4d+6) -\sum_j \rho(p_j) )
    [\Pbb^0]
\]
by \eqref{eqn:PieneMather} (i.e., \cite[p.~25]{MR1074588}). Comparing
with~\eqref{eqn:sdsur}, we see that
\[
\sum_j \rho(p_j) = 2d(d-3) + d(d^2-4d+6) = d^2(d-2)
\]
as stated.
\end{proof}

\begin{corol}
A self-dual surface in $\Pbb^3$ of degree $d$ with only ordinary
double points as singularities must have $d^2(d-2)/2$ singular
points. (In particular, $d$ is necessarily even.)
\end{corol}

The Kummer surface is an example. 

Analogous results may be obtained in higher dimension, by the same
method. For hypersurfaces, one finds the following constraint.

\begin{prop}\label{prop:hypcons}
The singular locus of a self-dual hypersurface of degree $d\ge 3$
in~$\Pbb^n$ has dimension $\ge \frac{n-3}2$.
\end{prop}

\begin{proof}
Let $X$ be a self-dual hypersurface of degree $d$, and assume that
$\dim \Sing X<\frac{n-3}2$; that is, $\dim \Sing X\le \frac{n-1}2-2$
if $n$ is odd, and $\dim \Sing X\le \frac n2-2$ if $n$ is even.  We
will show that this forces $d\le 2$. As we saw in
Example~\ref{exa:aiforhyp}, under the stated conditions on $\Sing X$ 
we must have
\begin{multline*}
\cmam(X) = d \{\Pbb^{n-1}\} + d(d-1) \{ \Pbb^{n-2} \} + \cdots \\ 
+ d(d-1)^{\frac {n-1}2-1} \{ \Pbb^{\frac {n-1}2+1} \} + d(d-1)^{\frac
  {n-1}2} \{ \Pbb^{\frac {n-1}2} \} + d(d-1)^{\frac {n-1}2+1} \{
\Pbb^{\frac {n-1}2-1} \} \\ 
+ a_{\frac {n-1}2-2} \{\Pbb^{\frac {n-1}2-2}\}+\cdots+ a_0 \{\Pbb^0\}
\end{multline*}
if $n$ is odd, and
\begin{multline*}
\cmam(X) = d \{\Pbb^{n-1}\} + d(d-1) \{ \Pbb^{n-2} \} + \cdots \\
+ d(d-1)^{\frac n2-1} \{ \Pbb^{\frac n2} \} + d(d-1)^{\frac n2} \{
\Pbb^{\frac n2-1} \} \\
+ a_{\frac n2-2} \{\Pbb^{\frac n2-2}\}+\cdots+ a_0 \{\Pbb^0\}
\end{multline*}
if $n$ is even, for some integers $a_i$. If $X$ is self-dual, by
Theorem~\ref{thm:basis} necessarily $a_i=d(d-1)^i$ and
\[
\begin{cases}
d(d-1)^{\frac {n-1}2-1}=d(d-1)^{\frac {n-1}2+1}\quad & \text{if $n$ is odd,} \\
d(d-1)^{\frac n2-1} = d(d-1)^{\frac n2} \quad & \text{if $n$ is even.}
\end{cases}
\]
Since
\[
d(d-1)^{\frac {n-1}2+1} = d(d-1)^{\frac {n-1}2-1} \iff d^2
(d-1)^{\frac {n-1}2-1}(d-2)=0
\]
and
\[
d(d-1)^{\frac n2} = d(d-1)^{\frac n2-1} \iff d(d-1)^{\frac n2-1}(d-2)=0
\]
we see that this cannot occur for $d\ge 3$.
\end{proof}

For example, the singular locus of a self-dual hypersurface of degree
$d\ge 3$ in $\Pbb^6$ or~$\Pbb^7$ has dimension at least~$2$. One
example in $\Pbb^7$ is the {\em Coble quartic,\/} which is singular
along a $3$-dimensional Kummer variety (\cite{MR1935152}). Applying
Theorem~\ref{thm:invo} shows that the Chern-Mather class of the Coble
quartic must be
\[
4[\Pbb^6]+16[\Pbb^5]+48[\Pbb^4]+(R+32)[\Pbb^3]+(4R-136)[\Pbb^2]+
(6R-288)[\Pbb^1] +(4R-208)[\Pbb^0]
\]
for some integer $R$. (What is $R$?) The hypersurface $x_0^n=x_1\cdots
x_n$ is self-dual for every $n$, and its singular locus has dimension
$n-3$.

In arbitrary codimension, one simple constraint holds for proper
subvarieties of {\em even\/} dimensional projective space. Let
$\cma(V)_j$ denote the degree of the $j$-dimensional component of
$\cma(V)$.

\begin{prop}
Let $V$ be a self-dual subvariety of $\Pbb^n$ with $n$ even. 
\newline Then $\sum_{j=0}^{\dim V} (-1)^j \cma(V)_j=0$.
\end{prop}

\begin{proof}
Let $\cma(V)=q(H)\cap [\Pbb^n]$ with $\deg q(H)\le n$. Then
$\sum_{j=0}^{\dim V} (-1)^j \cma(V)_j =(-1)^n q(-1)$. By self-duality
and Theorem~\ref{thm:invo}, $\cJ_n(q) = q$. It is immediately seen
that $\cJ_n(q)(-1) = (-1)^{n+1} q(-1)$, and the result follows.
\end{proof}

\subsection{Dual defect and the Katz-Kleiman-Holme formula}\label{ss:dude}
The (dual) {\em defect\/} of a variety $V\subseteq\Pbb^n$ is $\defe
V:=n-1-\dim V^\vee$: so that the defect of $V$ is $0$ when $V^\vee$ is
a hypersurface. It is known that nonlinear nonsingular complete
intersections have $0$ defect (\cite[p.~362]{MR0568897}; and see 
\cite[Proposition~3.1]{MR853445}).  In the nonsingular case, 
the Katz-Kleiman-Holme formula
\cite[Theorem~6.2]{MR2113135} shows that the defect is determined by
the total Chern class of $V$.  Theorem~\ref{thm:invo} implies that in
general the defect of $V$ is determined by its Chern-Mather class, and
gives an effective way to compute the defect and degree of~$V^\vee$.

\begin{prop}\label{prop:definvo}
Let $V$ be a proper closed subvariety of $\Pbb^n$, and let
$\cma(V)=q(H)\cap [\Pbb^n]$ for a polynomial $q$ of degree $\le n$ in
the hyperplane class $H$ in $\Pbb^n$.  Then $\codim V^\vee$ equals the
order of vanishing of $\cJ_n(q)$ at $0$, and $\deg V^\vee$ equals the 
absolute value of its trailing coefficient.
\end{prop}

\begin{proof}
This is an immediate consequence of Theorem~\ref{thm:invo}, since
$\cJ_n((-1)^{\dim V} q) =(-1)^{\dim V} \cJ_n(q)$.
\end{proof}

\begin{example}
(Cf.~\cite[Example~4.2]{MR2336689}.)  Consider the hypersurface $X$ of
  $\Pbb^4$ with equation
\[
x_0^2x_1+x_0x_2x_4+x_3x_4^2 = 0
\]
(this threefold and its dual are studied in \cite[p.~78--81]{MR1876644}). 
The algorithm in~\cite{harrisFMP} may be used to compute 
the Chern-Mather class of $X$ from this equation:
\[
\cma(X) = (3H+8H^2+9H^3+6H^4)\cap [\Pbb^4]\quad,
\]
so that $\cmam(X) = -(3H+8H^2+9H^3+6H^4)\cap [\Pbb^4]$. 
Applying $\cJ_4$, we get
\[
\cmam(X^\vee)=(3H^2+5H^3+4H^4)\cap [\Pbb^4]
\]
and we conclude that $X^\vee$ is a cubic surface.
\qede\end{example}

The same result may be recast in terms closer to the
Katz-Kleiman-Holme formula, by using Theorem~\ref{thm:basis}. If
$\cmam(V)=\sum_{i=0}^{n-1} a_i \{\Pbb^i\}$, then $\cmam(V^\vee) =
\sum_{i=0}^{n-1} a_{n-1-i} \{\Pbb^i\}$ by Theorem~\ref{thm:basis}, and
it follows that $\defe V$ equals the minimum $i$ such that $a_i\ne 0$,
and $\deg V^\vee=a_{\defe V}$. The ranks~$a_i$ were computed in
Proposition~\ref{prop:ai}, and we get the following consequence 

\begin{prop}\label{prop:KKH}
Let $V$ be a proper closed subvariety of $\Pbb^n$, and let
\begin{equation}\label{eqn:KKH}
a_i=\sum_{j=i}^{\dim V} \binom{j+1}{i+1} (-1)^{\dim V-j} \cma(V)_j
\end{equation}
where $\cma(V)_j$ is the degree of the $j$-dimensional component of
$\cma(V)$.  Then $\defe V$ equals the minimum $i$ such that $a_i\ne
0$, and $\deg V^\vee=a_{\defe V}$.
\end{prop}

In the nonsingular case, this expression may be found 
in~\cite[formula (3)]{MR943533}. In this reference it is stated that 
`The full content of formula (3) does not seem to have an immediate 
generalization to the singular case'. Proposition~\ref{prop:KKH} provides 
such a generalization. The first full treatment of the singular case was 
provided by R.~Piene (\cite{polarpiene}), who should be credited with 
the realization that, in dealing with such issues, Chern-Mather classes play 
the same role in the singular case as ordinary Chern classes in the 
nonsingular case (cf.~\cite{MR1074588}).

\begin{example}
Consider the (singular) surface $S$ with equation
\[
16x_0^2x_3^3-8x_0x_1^2x_3^2+36x_0x_1x_2^2x_3-27x_0x_2^4
+x_1^4x_3-x_1^3x_2^2 = 0
\]
in $\Pbb^3$. Again using the algorithm in~\cite{harrisFMP}, we get
\[
\cma(S) = \cmam(S)=5[\Pbb^2]+11[\Pbb^1]+7[\Pbb^0]\quad.
\]
From~\eqref{eqn:KKH} we get $a_0=7-2\cdot 11+3\cdot 5 = 0$, $a_1=-11+3\cdot 5 =4$,
$a_2=5$, and we conclude that the dual of $S$ is a quartic curve.
\qede\end{example}

\begin{remark}\label{rem:MT}
A formula for the degree of the dual variety at the same level of
generality as Proposition~\ref{prop:KKH}, and also using Chern-Mather
classes, is given in~\cite[Theorem~3.4]{MR2336689}: according to this
reference,
\begin{equation}\label{eqn:MT}
\deg V^\vee = (-1)^{\dim V+r+1} \sum_{j=0}^{r-1} \binom{r+1}j(r-j) \int_V 
\frac{H^j}{(1+H)^{r+1}}\cap \cma(V)
\end{equation}
where $r=\codim V^\vee=\defe V+1$. At first blush this does not appear
to be equivalent to the statement of Proposition~\ref{prop:KKH}: for
example, if $\dim V=3$ and $\defe V^\vee=2$, then with
$c_i:=\cma(V)_i$ we have
\[
a_2 = -c_2 + 4 c_3
\]
according to~\eqref{eqn:KKH}, while the right-hand side of~\eqref{eqn:MT} equals
\[
-3 c_0+4 c_1 - 4 c_2 + 4 c_3
\]
for $r=3$.  However, it can be shown that {\em if
  $a_0=\dots=a_{r-2}=0$, then\/} the right-hand side of~\eqref{eqn:MT}
does equal $a_{r-1}$. With $\dim V=3$, we have
\[
a_0 = -c_0+2 c_1-3 c_2+4c_3\quad, \quad a_1 = c_1-3 c_2+6 c_3\quad,
\]
and indeed $-3 c_0+4 c_1 - 4 c_2 + 4 c_3 = -c_2 + 4 c_3$ if $a_0=a_1=0$.

Expression~\eqref{eqn:KKH} appears to be somewhat more efficient
than~\eqref{eqn:MT}, since it only involves terms of the Chern-Mather
class of dimension $\ge \defe V$.  
\qede\end{remark}

\subsection{Chern-Mather classes of cones and local Euler obstruction 
at cone points}\label{ss:cones}
Fix a subspace $\Pbb^m$ of $\Pbb^n$, with $m<n$, and a complementary
subspace $\Lambda=\Pbb^{n-m-1}$. Let $W$ be a proper subvariety of
$\Pbb^m$, and let $V$ be the cone over~$W$ with vertex
$\Lambda$. Duality considerations allow us to express the Chern-Mather
class of~$V$ in terms of the Chern-Mather class for $W$, by means of
the involution introduced in~\S\ref{sec:intro}.

\begin{prop}\label{prop:conform}
Let $q_W$, resp., $q_V$ be polynomials of degree less than $m$ resp.,
$n$ such that
\[
\cma(W)=q_W(H)\cap [\Pbb^m]\quad,\quad \cma(V)=q_V(H)\cap [\Pbb^n]\quad.
\]
Then
\[
q_V=\cJ_n((-H)^{n-m}\cdot \cJ_m (q_W))\quad.
\]
\end{prop}

The hypothesis that the vertex $\Lambda$ be contained in a
complementary subspace can be relaxed; this will be pointed out at the
end of the section.

\begin{proof}
The subspace dual to $\Lambda\cong \Pbb^{n-m-1}$ in $\Pbb^{n\vee}$ is
naturally identified with $\Pbb^{m\vee}$, and it is easy to see that
$V^\vee$ coincides with $W^\vee$ in this subspace. (See e.g.,
\cite[Theorem~1.23]{MR2113135} for the case $m=n-1$, from which the
general case is easily derived.) By Theorem~\ref{thm:invo},
\[
\cmam(V^\vee) = \cJ_n((-1)^{\dim V} q_V)\cap [\Pbb^{n\vee}] 
\]
and
\[
\cmam(W^\vee) = \cJ_m((-1)^{\dim W} q_W)\cap [\Pbb^{m\vee}] = (-1)^{\dim W}
H^{n-m}\cdot \cJ_m(q_W)\cap [\Pbb^{n\vee}]\quad.
\]
Since these two classes coincide and $\cJ_n$ is an involution, 
\[
q_V= \cJ_n^2 (q_V) = (-1)^{\dim V-\dim W}\cJ_n (H^{n-m} \cdot
\cJ_m(q_W))\quad.
\]
The statement follows, since $\dim V-\dim W=n-m$.
\end{proof}

Explicitly, 
\begin{equation}\label{eqn:cone}
q_V(H)=(1+H)^{n-m} q_W(H) + (-1)^m q_W(-1)\, H^{m+1}
\left((1+H)^{n-m}-H^{n-m}\right) \quad.
\end{equation}
There are analogous formulas expressing the {\em
  Chern-Schwartz-MacPherson\/} (CSM) class of $V$ in terms of the
class of the CSM of $W$, and more generally the CSM class of the join
of two subvarieties $W_1$, $W_2$ in complementary subspaces in terms
of the CSM classes of $W_1$ and $W_2$ 
(\cite[Theorem~3.13]{MR2782886}; also cf.~\cite{MR1078099}). Putting 
these pieces of information together yields the following formula for the 
local Euler obstruction of a cone singularity.

\begin{prop}\label{prop:euler}
Let $V$ be a cone over a subvariety $W\subseteq \Pbb^m$ with vertex
$\Lambda$ in a complementary subspace (as above). Let $p\in
\Lambda$. Then the local Euler obstruction $\Eu_V(p)$ equals
\[
\Eu_V(p)=\sum_{j=0}^{\dim W} (-1)^j \cma(W)_j
\]
where $\cma(W)_j$ denotes the degree of the $j$-dimensional component
of $\cma(W)$.
\end{prop}

Again, it is not necessary to require $\Lambda$ to be contained in a
complementary subspace; see Proposition~\ref{prop:Piene}. 
Proposition~\ref{prop:euler} is a consequence of
Proposition~\ref{prop:conform}, by means of a lemma which seems worth
stating explicitly. In the situation described above, let
$\pi:\Pbb^n\smallsetminus \Lambda \to \Pbb^m$ be the projection, and
let $\varphi$ be a constructible function defined on $\Pbb^m$.  Let
$\pi^*\varphi$ be the constructible function on $\Pbb^n$ defined by
\[
\pi^*\varphi(p) =
\begin{cases}
0\quad &\text{if $p\in \Lambda$} \\
\varphi(\pi(p)) \quad &\text{if $p\not\in \Lambda$}
\end{cases}
\]
Also, let $q_\varphi$, resp., $q_{\pi^*\varphi}$ be the polynomials of
degrees $\le m$, resp., $\le n$ such that
\[
c_*(\varphi)=q_\varphi(H)\cap [\Pbb^m]\quad,\quad
c_*(\pi^*\varphi)=q_{\pi^*\varphi}(H)\cap [\Pbb^n]\quad.
\]
Here $c_*$ is MacPherson's natural transformation \cite{MR0361141}.
For example $c_*(\one_W)=\csm(W)$ and $c_*(\Eu_W)=\cma(W)$
(cf.~\S\ref{sec:mainpf}).

\begin{lemma}\label{lem:eulerpb}
With the above notation,
\begin{equation}\label{eqn:csmcone}
q_{\pi^*\varphi}(H) = (1+H)^{n-m} q_\varphi(H)\quad.
\end{equation}
\end{lemma}

\begin{proof}
By linearity we may assume that $\varphi=\one_W$ for a subvariety $W$
of $\Pbb^m$; and therefore $\pi^*\varphi=\one_V-\one_\Lambda$, where
$V$ is the cone over $W$ as above. By \cite[Theorem~3.13]{MR2782886},
the polynomial corresponding to $\csm(V)=c_*(\one_V)$ equals
\[
((q_\varphi(H)+H^{m+1})(1+H)^{n-m})(1+H)^{n-m}-H^{n+1}\quad.
\]
It follows that $q_{\pi^*\varphi}$ is obtained from this polynomial by
subtracting the polynomial for~$\one_\Lambda$, i.e.,
$H^{m+1}((1+H)^{n-m}-H^{n-m})$, and this yields the stated formula.
\end{proof}

A good interpretation (and, with due care, an alternative argument)
for identity~\eqref{eqn:csmcone} is that MacPherson's natural
transformation preserves products (\cite{MR1158750},
\cite{MR2209219}), and the factors $1+H$ may be viewed as the $\csm$
class of a linearly embedded affine line.

\begin{proof}[Proof of Proposition~\ref{prop:euler}]
Recall that $\cma(V)=c_*(\Eu_V)$ (\S\ref{sec:mainpf}). The cone $V$ is
the union of the vertex $\Lambda$ and the inverse image
$\pi^{-1}(W)$. The restriction of $\Eu_V$ to $\pi^{-1}(W)$ equals
$\pi^*\Eu_W$; this follows for example from \cite[p.~426,
  3.]{MR0361141}. The restriction of $\Eu_V$ to $\Lambda$ is constant,
so it equals $\Eu_V(p)$. Thus
\[
\Eu_V = \pi^*\Eu_W + \Eu_V(p)\, \one_\Lambda\quad,
\]
and therefore
\begin{equation}\label{eqn:cmaV}
\cma(V) = c_*(\pi^*\Eu_W)+\Eu_V(p)\, c_*(\one_\Lambda)\quad.
\end{equation}
Now $\Lambda\cong\Pbb^{n-m-1}\subseteq\Pbb^n$, therefore
\[
c_*(\one_\Lambda) = c(T\Pbb^{n-m-1})\cap [\Lambda] =
\left((1+H)^{n-m}-H^{n-m}\right) H^{m+1}\cap [\Pbb^n]\quad.
\]
On the other hand, by Lemma~\ref{lem:eulerpb},
\[
c_*(\pi^*\Eu_W)=(1+H)^{n-m} q_W(H)\quad.
\]
Therefore, identity~\eqref{eqn:cmaV} states that
\[
q_V(H)=(1+H)^{n-m} q_W(H) + \Eu_V(p)\,H^{m+1}
\left((1+H)^{n-m}-H^{n-m}\right) \quad.
\]
Comparing this identity with~\eqref{eqn:cone} proves that
$\Eu_V(p)=(-1)^m q_W(-1)$. Finally, $q_W(H)=\sum_j \cma(W)_j H^{m-j}$,
hence
\[
(-1)^m q_W(-1)=(-1)^m \sum_{j\ge 0} \cma(W)_j (-1)^{m-j}=\sum_{j\ge 0} (-1)^j \cma(W)_j
\]
as needed.
\end{proof}

\begin{example}\label{exa:MP2}
The local Euler obstruction at the vertex of a cone in $\Pbb^n$ over a
nonsingular curve of degree $d$ and genus $g$ in $\Pbb^{n-1}$ equals
$2-2g-d$. Indeed, the Chern-Mather class of a nonsingular curve equals
the Chern class of its tangent bundle, so it pushes forward to
$d[\Pbb^1] + (2-2g) [\Pbb^0]$, and the formula follows from
Proposition~\ref{prop:euler}.

If the curve is a plane curve, then $2-2g=3d-d^2$, so the local Euler
obstruction at the vertex of a cone in $\Pbb^3$ over a nonsingular
plane curve of degree~$d$ equals $2d-d^2$, in agreement
with~\cite[p.~426, 2.]{MR0361141}.  \qede\end{example}

It is natural to ask whether the hypothesis that the vertex $\Lambda$
be in a complementary subspace may be relaxed. In the situation
presented at the beginning of this subsection, the codimension of $W$
is bounded below by $\dim\Lambda+2$. It turns out that this is the
only requirement needed for the results, provided
that $\Lambda$ is general.

\begin{prop}\label{prop:Piene}
Let $W\subseteq \Pbb^n$ be a closed subvariety of codimension $\ge
r\ge 2$, and let $\Lambda \cong\Pbb^{r-2}\subseteq \Pbb^n$ be a
general subspace. Let $V$ be the cone over $W$ with vertex
$\Lambda$. Also, let $q_W$, resp., $q_V$ be polynomials of degree less
than $n-r-1$, resp., $n$ such that
\[
\cma(W)=H^{r-1}\,q_W(H)\cap [\Pbb^n]\quad,\quad \cma(V)=q_V(H)\cap
    [\Pbb^n]\quad.
\]
Then
\[
q_V=\cJ_n((-H)^{r-1}\cdot \cJ_{n-r+1} (q_W))
\]
and $\Eu_V(p)=\sum_{j=0}^{\dim W} (-1)^j \cma(W)_j$ for all $p\in\Lambda$.
\end{prop}

\begin{proof}
The Chern-Mather class is preserved under general projections
(\cite[Corollaire, p.~20]{MR1074588}). Therefore,
$\cma(W)=\cma(\pi(W))$ as classes in~$\Pbb^n$, where $\pi:
\Pbb^n\smallsetminus \Lambda \to \Pbb^{n-\dim\Lambda-1}$ is the
projection ($W\cap \Lambda=0$ by dimension considerations, since
$\Lambda$ is general). Since the cone over $W$ with vertex $\Lambda$
equals the cone over $\pi(W)$ with vertex~$\Lambda$, we may then
replace $W$ with~$\pi(W)$. This reduces the general statement to the
case in which $\Lambda$ is a complementary subspace, so the stated
formulas follow from Propositions~\ref{prop:conform}
and~\ref{prop:euler}.
\end{proof}

\begin{example}
Let $C$ be a twisted cubic in $\Pbb^3$, $p$ a general point of
$\Pbb^3$, and let $V$ be the cone over $C$ with vertex at $p$. Then
Proposition~\ref{prop:Piene} gives $ \cma(V) =
3[\Pbb^2]+5[\Pbb^1]+[\Pbb^0] $ (which is confirmed by an explicit
computation performed with the algorithm in~\cite{harrisFMP}) and
$\Eu_V(p)=-1$.  
\qede\end{example}

\end{document}